\documentclass[conference]{IEEEtran}
\IEEEoverridecommandlockouts

\usepackage{cite}
\usepackage{amsmath,amssymb,amsfonts,mathtools}
\usepackage{algorithmic}
\usepackage{graphicx}
\usepackage{textcomp}
\usepackage[dvipsnames]{xcolor}
\usepackage{tikz}
\usepackage{tikzscale}
\usepackage{tikz-3dplot}
\usepackage{pgf}
\usepackage{pgfplots}
\usepackage{pgfplotstable}
\usepackage{pdflscape}
\usepackage{subcaption}
\usepackage{cuted}   % for full-width equations inside two-column IEEE papers
\usepackage{url}

\usepackage{mathsemantics}
\usepackage[commentmarkup = footnote]{changes}
\usepackage{todonotes}
\usetikzlibrary{decorations.pathmorphing,arrows.meta}

\allowdisplaybreaks[2]
\usepackage{changepage}

\makeatletter
\def\widebreve{\mathpalette\wide@breve}
\def\wide@breve#1#2{\sbox\z@{$#1#2$}%
  \mathop{\vbox{\m@th\ialign{##\crcr
  \kern0.08em\brevefill#1{0.8\wd\z@}\crcr\noalign{\nointerlineskip}%
  $\hss#1#2\hss$\crcr}}}\limits}
\def\brevefill#1#2{$\m@th\sbox\tw@{$#1($}%
  \hss\resizebox{#2}{\wd\tw@}{\rotatebox[origin=c]{90}{\upshape(}}\hss$}
\makeatother

\newcommand{\R}{\mathbb{R}}
\renewcommand{\so}{\mathfrak{so}}
\newcommand{\se}{\mathfrak{se}}

\newcommand{\Exp}{\normalfont{\text{Exp}}}

\newcommand{\SO}{\normalfont{\text{SO}}}
\newcommand{\SE}{\normalfont{\text{SE}}}

\renewcommand{\riemannian}{\dual}

\newtheorem{theorem}{Theorem}[section]
\newtheorem{lemma}[theorem]{Lemma}
\newenvironment{proof}{\textit{Proof:}}{\hfill$\square$}

\newtheorem{ definition}[theorem]{Definition}
\newtheorem{remark}[theorem]{Remark}

\begin{document}

\title{Riemannian Metric Preconditioning for Trajectory Tracking\\
\thanks{J. Goodman was supported by the Marie Skłodowska-Curie grant agreement No. 101206748 (GNACS).}
}

\author{\IEEEauthorblockN{Jacob Goodman}
\IEEEauthorblockA{\textit{Department of Mathematical Sciences} \\
\textit{Norwegian University of Science and Technology}\\
Trondheim, Norway \\
jacob.goodman@ntnu.no}
\and
\IEEEauthorblockN{Hajg Jasa}
\IEEEauthorblockA{\textit{Department of Mathematical Sciences} \\
\textit{Norwegian University of Science and Technology}\\
Trondheim, Norway \\
hajg.jasa@ntnu.no}
}
\date{}
                       
\maketitle

\begin{abstract}
We introduce a rank-one Riemannian cometric update inducing a modification of the Riemannian metric that makes specific directions of motion cheaper to travel along.
We establish basic completeness properties of this reward metric, and we give an explicit characterization of its Levi--Civita connection.
We propose a preconditioned trajectory-tracking strategy by adding the connection-difference term to a standard intrinsic PD control, and illustrate the construction on a connection control-affine system on the Special Euclidean group with a maze navigation experiment.
When the nominal trajectory is an integral curve of the vector field used to define the reward metric, our methodology improves the overall tracking.
\end{abstract}


\section{Introduction}

Trajectory generation and tracking for mechanical systems are naturally formulated on nonlinear configuration spaces rather than in Euclidean coordinates. This is especially true for systems evolving on manifolds and Lie groups, such as rigid bodies on \(\SE(3)\), where intrinsic formulations avoid singular parameterizations and preserve the underlying geometric structure. In geometric mechanics and control, the Riemannian metric, Levi--Civita connection, and geodesic flow play a central role in the modeling and analysis of simple mechanical systems, and related ideas have been used in intrinsic tracking and stabilization on Lie groups \cite{BulloLewis}.

A key observation is that the choice of Riemannian metric is not merely a modeling convenience: it determines the local notion of kinetic energy, the corresponding Levi--Civita connection, and hence the free second-order dynamics of the system. This has long been recognized in the controlled Lagrangians literature \cite{Bloch}, where closed-loop stabilization is achieved by reshaping the kinetic and potential structure of the underlying mechanical system, as well as through the use of artificial potential functions for navigation \cite{Rimon}. More recently, path-planning methods have emphasized that Riemannian metrics can also be used as design variables, for example, to avoid forbidden regions or encode directional costs \cite{Klein, Martin}. These developments suggest that metric design can serve not only as a modeling tool, but also as a mechanism for biasing closed-loop behavior in mechanical control problems.

In this paper, given a Riemannian metric $g$ and a one-form $\mu$, we consider the rank-one update
$$g_\mu = g - \frac{\mu \otimes \mu}{1+\|\mu^\sharp\|_g^2}.$$
Unlike the more familiar penalty metric $g+\mu\otimes\mu$, which inflates distances in the direction of $\mu^\sharp$, the metric \(g_\mu\) contracts them. Equivalently, motion tangent to \(\mu^\sharp\) becomes cheaper, while directions \(g\)-orthogonal to \(\mu^\sharp\) are unchanged. We therefore view \(g_\mu\) as a \emph{reward metric} that encodes preferred directions of travel.

Our motivation comes from trajectory tracking. Suppose a nominal trajectory \(q_d\) is given, together with a vector field \(X\) whose integral curves approximate the desired motion. Choosing \(\mu^\sharp=\lambda X\) with \(\lambda\ge 0\) yields a family of metrics \(g_\lambda\) for which motion tangent to \(X\) becomes cheaper as \(\lambda\) increases. This suggests a geometric preconditioning strategy: design feedback relative to the modified Levi--Civita connection, and then pull the resulting dynamics back to the original system through the difference tensor of the two connections. The outcome is an intrinsic correction term that augments a standard proportional-derivative tracking law. Our aim in this work is to establish the basic geometric properties of this metric modification strategy, derive the exact correction term appearing in the dynamics, and investigate how this correction affects trajectory tracking. The analysis and simulations show that the preconditioner is sensitive to the local behavior of the vector field used to deform the metric, and suggest the need for further investigations to streamline design and provide rigorous safety guarantees.

The main contributions are as follows. $1)$ We define the reward metrics $g_\mu$ as the duals of penalty cometrics. $2)$ We give a sufficient condition under which completeness of $g$ transfers to $g_\mu$. $3)$ We derive an explicit formula for the difference tensor between the Levi--Civita connections of $g$ and $g_\mu$. $4)$ We formulate a preconditioned intrinsic trajectory-tracking law for affine connection control systems on Riemannian manifolds. $5)$ We apply the construction to an underactuated rigid-body system on $\mathrm{SE}(3)$ and compare the resulting preconditioned controller with an unmodified geometric tracking controller in a maze navigation problem. The numerical experiments include both a reference trajectory that is an integral curve of the vector field defining the metric modification and a reference trajectory that is not, illustrating both the benefits and limitations of the method.

The paper is organized as follows. Section~II reviews the required geometric preliminaries. Section~III introduces the metric/cometric modification, proves the completeness result, and derives the associated Levi--Civita correction term. Section~IV formulates the preconditioned tracking strategy for affine connection control systems and specializes the notation to the $\mathrm{SE}(3)$ model used in the application. Section~V presents the numerical maze-navigation experiments, including tracking errors and control-effort comparisons for Euclidean and preconditioned feedback laws. Section~VI concludes with limitations and directions for future work.

\section{Preliminaries}

Let \(\cM\) be a finite-dimensional smooth manifold. A Riemannian metric \(g\) on \(\cM\) is a symmetric positive-definite \((0,2)\)-tensor field, which measures lengths and angles. It induces the musical morphisms \(\cdot^\flat \colon \tangentBundle \to \cotangentBundle\) and \(\cdot^\sharp \colon \cotangentBundle \to \tangentBundle\), characterized by
\[
X^\flat(Y)=g(X,Y),
\qquad
(X^\flat)^\sharp=X
\]
for all \(p\in\cM\) and \(X,Y\in \tangentSpace{p}\). This extends pointwise to vector fields and one-forms.

Equivalently, one may specify a symmetric positive-definite \((2,0)\)-tensor field \(g^*\), called a \emph{cometric}, related to \(g\) by
\[
g(X,Y)=g^*(X^\flat,Y^\flat).
\]
If \(g_{ij}\) is the coordinate matrix of \(g\), then the coordinate matrix of \(g^*\) is given by the inverse matrix \(g^{ij}\).

A choice of metric \(g\) determines geodesics, namely curves satisfying \(\nabla_{\dot\gamma}\dot\gamma=0\), where \(\nabla\) is the Levi--Civita connection of \((\cM,g)\). The (restricted) exponential map at \(p\in\cM\) is the map \(\exponential{p} \colon \tangentSpace{p}\to\cM\) defined by \(\exponential{p}(X_p)=\gamma(1)\), where \(\gamma\) is the geodesic starting at \(p\) with initial velocity \(X_p\). A \emph{retraction} at \(p\) is a first-order approximation \(\retract{p}\colon\tangentSpace{p}\to\cM\) of the exponential map satisfying
\[
\retract{p}(0_p)=p,
\qquad
\d_{0_p}\retract{p}=\id_{\tangentSpace{p}}.
\]
These conditions imply that \(\retract{p}\) is locally invertible near \(0_p\).
See, \eg, \cite[Chapter~4]{Absil} for more details. We shall also use a vector transport to compare tangent vectors based at nearby points. By a vector transport we mean a smooth assignment $\mathrm{VT}_{x\to y} \colon \tangentSpace{x} \to \tangentSpace{y}$,
defined for $x,y\in\mathcal M$ sufficiently close, such that
$\mathrm{VT}_{x\to x} = \operatorname{id}_{\tangentSpace{x}}$ and
$\mathrm{VT}_{x\to y}$ is linear on each tangent space. A natural choice of vector transport associated with a retraction is the differentiated retraction: if $y=\operatorname{retr}_x(\eta)$ and $\eta=\operatorname{retr}_x^{-1}(y)$, then we take
\[
    \mathrm{VT}^{\operatorname{retr}}_{x\to y}\xi
    \coloneq
    (\d\,\operatorname{retr}_x)_\eta[\xi]
    =
    \d_\eta\,\operatorname{retr}_x[\xi]
    ,
    \qquad
    \xi \in \tangentSpace{x},
\]
where we identify $\tangentSpace{\eta}[(\tangentSpace{x})]$ with $\tangentSpace{x}$. This
transport is not generally an isometry. When the retraction is chosen to
be the Riemannian exponential map, another canonical choice is
Levi--Civita parallel transport along the corresponding geodesic from
$x$ to $y$. In Euclidean space, both choices reduce to the identity
transport.

A \emph{Lie group} is a smooth manifold \(G\) equipped with a smooth group structure. Its Lie algebra is \(\mathfrak g \coloneq \tangentSpace{e}[G]\), where \(e\in G\) is the identity. The group exponential map \(\Exp:\mathfrak g\to G\) is defined by \(\Exp(X)=\gamma(1)\), where \(\gamma\) is the integral curve of \(X\in\mathfrak g\) starting at \(e\). In general, the Riemannian and group exponential maps do not coincide at \(e\).

\section{Metric Modification}

Let $g$ be a Riemannian metric on $\cM$ with cometric $g^\ast$, and \(\mu\in\Omega^1(\cM)\). We consider the penalty cometric
\[
g_\mu^*=g^*+\mu^\sharp\otimes\mu^\sharp.
\]
In coordinates,
\[
(g_\mu^*)^{ij}=g^{ij}+\mu^i\mu^j,
\qquad
\mu^i=g^{ij}\mu_j.
\]
By the Sherman--Morrison formula, the inverse metric has components
\[
(g_\mu)_{ij}
=
g_{ij}
-
\frac{g_{i\ell}\mu^\ell\mu^m g_{mj}}{1+\mu^r g_{rs}\mu^s}
=
g_{ij}
-
\frac{\mu_i\mu_j}{1+\mu_s\mu^s},
\]
and hence
\[
g_\mu
=
g-\frac{\mu\otimes\mu}{1+\|\mu^\sharp\|_g^2}.
\]
\begin{remark}
\label{rm:modification-via-vector}
Equivalently, for a vector field \(X\in\mathfrak X(\cM)\),
\begin{equation}
\label{eq:metric-modification-vector-field}
g_X
=
g-\frac{X^\flat\otimes X^\flat}{1+\|X\|_g^2},
\end{equation}
where \(X^\flat\) is taken with respect to the original metric \(g\).
\end{remark}

We refer to metrics of this form as \emph{reward metrics}, contrasting the more usual penalty metrics $g = g + \mu \otimes \mu$ which often appear in subriemannian geometry and nonholonomic mechanics, and, more recently, in obstacle avoidance strategies for control systems.

The modification \(g+\mu\otimes\mu\) inflates distances in the direction of \(\mu^\sharp\), while
\[
g_\mu(X,X)
=
\|X\|_g^2-\frac{\mu(X)^2}{1+\|\mu^\sharp\|_g^2}
=
\|X\|_g^2-\frac{g(\mu^\sharp,X)^2}{1+\|\mu^\sharp\|_g^2}
\le \|X\|_g^2.
\]
Thus \(\|X\|_{g_\mu}=\|X\|_g\) whenever \(g(\mu^\sharp,X)=0\), and otherwise \(\|X\|_{g_\mu}<\|X\|_g\). Hence, \(g_\mu\) appears to reward motion in the direction of \(\mu^\sharp\) relative to \(g\).

Before turning to control design, we establish a basic completeness property of \(g_\mu\). Recall that \(g\) is \emph{geodesically complete} if every unit-speed geodesic extends to all of \(\R\). 
By Hopf--Rinow, this is equivalent to completeness of the metric space \((\cM,d_g)\), where \(d_g\) is the Riemannian distance function associated to \(g\). The completeness of a Riemannian metric is preserved under inflations of distance, but not in general under contractions. Hence, \(g_\mu\) can be incomplete even when \(g\) is complete. The following theorem gives a sufficient condition preventing this by controlling the rate of contraction along sequences that escape to infinity, which is similar in spirit to Theorem~3.7.15 of \cite{AM}, which provides sufficient conditions for the completeness of mechanical systems.

\begin{theorem}
Let \(V_0\colon[0,\infty)\to(0,\infty)\) be continuous and nondecreasing such that
\[
\int_0^\infty \frac{dx}{\sqrt{V_0(x)}}=+\infty.
\]
If \(g\) is a complete Riemannian metric on \(\cM\), and \(\mu\in\Omega^1(\cM)\) satisfies
\[
1+\|\mu^\sharp(p)\|_g^2 \le V_0(d_g(p,x_0)),
\qquad
\forall p\in\cM,
\]
for some \(x_0\in\cM\), then
\[
g_\mu
=
g-\frac{\mu\otimes\mu}{1+\|\mu^\sharp\|_g^2}
\]
is a complete Riemannian metric on \(\cM\).
\end{theorem}

\begin{proof}
For any \(p\in\cM\) and \(X\in \tangentSpace{p}\),
\[
\|X\|_{g_\mu}^2
=
\|X\|_g^2-\frac{g(\mu^\sharp,X)^2}{1+\|\mu^\sharp\|_g^2}
\ge
\frac{\|X\|_g^2}{1+\|\mu^\sharp\|_g^2}.
\]
Thus \(g_\mu\) dominates the conformal metric \(fg\), where
\[
f(p)\coloneq\frac{1}{1+\|\mu^\sharp(p)\|_g^2}.
\]
It therefore suffices to show that \(fg\) is complete.

Suppose toward a contradiction that \(fg\) is geodesically incomplete, and let \(\gamma \colon [0,T)\to\cM\) be a maximal unit-speed geodesic of \(fg\). Then \(\gamma\) eventually leaves every compact subset of \(\cM\). Let \(\rho \coloneq d_g(\cdot,x_0)\). Since \(g\) is complete, closed bounded sets are compact, so \(\rho(\gamma(t))\to\infty\) as \(t\to T^-\). Moreover,
\begin{align*}
T
&=
L_{fg}[\gamma]
=
\int_0^T \|\dot\gamma(t)\|_{fg}\,dt
=
\int_0^T \|\dot\gamma(t)\|_g\sqrt{f(\gamma(t))}\,dt \\
&\ge
\int_0^T \frac{\|\dot\gamma(t)\|_g}{\sqrt{V_0(\rho(\gamma(t)))}}\,dt.
\end{align*}

Let \(N\coloneq\lceil \rho(\gamma(0))\rceil\), and for each integer \(n>N\) define
\[
t_n\coloneq\inf\{t\in[0,T)\mid \rho(\gamma(t))=n\}.
\]
Then \(t_n\to T\), and \(\rho(\gamma(t))\le n+1\) for all \(t\in[t_n,t_{n+1}]\). Since \(V_0\) is nondecreasing,
\[
\frac{1}{\sqrt{V_0(\rho(\gamma(t)))}} \ge \frac{1}{\sqrt{V_0(n+1)}}
\qquad
\text{for } t\in[t_n,t_{n+1}].
\]
Hence,
\begin{align*}
T
&\ge
\sum_{n=N}^\infty
\int_{t_n}^{t_{n+1}}
\frac{\|\dot\gamma(t)\|_g}{\sqrt{V_0(\rho(\gamma(t)))}}\,dt \\
&\ge
\sum_{n=N}^\infty
\frac{1}{\sqrt{V_0(n+1)}}
\int_{t_n}^{t_{n+1}} \|\dot\gamma(t)\|_g\,dt \\
&\ge
\sum_{n=N}^\infty
\frac{d_g(\gamma(t_n),\gamma(t_{n+1}))}{\sqrt{V_0(n+1)}}
\ge
\sum_{n=N}^\infty \frac{1}{\sqrt{V_0(n+1)}} \\
&\ge
\int_{N+1}^\infty \frac{dx}{\sqrt{V_0(x)}},
\end{align*}
contradicting the divergence of the final integral.
\end{proof}

The theorem shows that \(g_\mu\) is complete whenever \(g\) is complete and \(\|\mu^\sharp(x)\|_g\) does not increase too fast as \(x\to\infty\). In particular, it suffices that \(\|\mu^\sharp\|_g=O(d_g(x_0,\cdot))\).

\subsection{Levi--Civita Connection}
We now turn our attention to the Levi--Civita connection of the new reward metric, and give an explicit characterization for it in terms of the original metric.
\begin{lemma}
\label{lemma:new-levi-civita}
Let $(\cM, g)$ be a Riemannian manifold, and let
\[
g_\mu = g - \frac{\mu \otimes \mu}{1 + \riemanniannorm{\mu^\sharp}^2}
\]
for some $\mu \in \Omega^1(\cM) \setminus \{0\}$.
Denote the Levi--Civita connection with respect to $g$ and $g_\mu$ by $\nabla$ and $\tilde{\nabla}$, respectively.
Set
\[
\Lambda \coloneqq 1 + \riemanniannorm{\mu^\sharp}^2,
\quad
\mu^\sharp_\perp
\coloneqq
\frac{\mu(X)\mu(Y)}{\Lambda}\,\mu^\sharp
-
\frac{\mu(Y)}{2}\,X
-
\frac{\mu(X)}{2}\,Y.
\]
Then, for all $X, Y \in \Gamma(\tangentBundle)$,
\begingroup\small
\begin{equation}
\begin{aligned}
\tilde{\nabla}_X Y
={}&
\nabla_X Y
\\
&
+
\left[
  \frac{1}{\Lambda}
  \riemannian{
    \mu(Y)\nabla_X \mu^\sharp + \mu(X)\nabla_Y \mu^\sharp
  }{\mu^\sharp}
  -
  \riemannian{\nabla_X \mu^\sharp}{Y}
\right]\mu^\sharp
\\
&
\hspace{-2.5em}
+
\left[
  \frac{1}{2}\,\d\mu(X,Y)
  -
  \frac{\mu(X)\mu(Y)}{\Lambda}
  \frac{\mu(\nabla_{\mu^\sharp}\mu^\sharp)}{\riemanniannorm{\mu^\sharp}^2}
  +
  \frac{\Lambda\,\d\mu(\mu^\sharp_\perp,\mu^\sharp)}{\riemanniannorm{\mu^\sharp}^2}
\right]\mu^\sharp
\\
&-
\frac{\mu(X)\mu(Y)}{\Lambda^2}
\left[
  \nabla_{\mu^\sharp}\mu^\sharp
  -
  \frac{\mu(\nabla_{\mu^\sharp}\mu^\sharp)}{\riemanniannorm{\mu^\sharp}^2}\,\mu^\sharp
\right]
\\
&+
\frac{1}{\Lambda}\,\d\mu(\mu^\sharp_\perp,\cdot)^\sharp
-
\frac{\d\mu(\mu^\sharp_\perp,\mu^\sharp)}{\riemanniannorm{\mu^\sharp}^2}\,\mu^\sharp .
\end{aligned}
\label{eq:levi-civita-connection}
\end{equation}
\endgroup
where $g = \riemannian{\cdot}{\cdot}$ and $\riemanniannorm{\cdot}$ is the norm induced by $g$.
\end{lemma}

\begin{proof}
Set $\Lambda \coloneqq 1 + \riemanniannorm{\mu^\sharp}^2$ for brevity. From the Koszul formula, we have
\begin{align*}
2g(\nabla_X Y, Z)
&=
X(g(Y,Z)) + Y(g(X,Z)) - Z(g(X,Y)) \\
&\quad
+ g([X, Y], Z) - g([X,Z], Y) - g([Y,Z], X).
\end{align*}
Applying the Koszul formula to $g_\mu$, and writing $\tilde{\nabla}_X Y = \nabla_X Y + T_X Y$ for some tensor $T$ to be determined, we find that
\begingroup\small
\begin{multline*}
2g(T_XY, Z)
-
\frac{2}{\Lambda}\,\mu(\nabla_X Y)\mu(Z)
-
\frac{2}{\Lambda}\,\mu(T_XY)\mu(Z)
=
\\
-
X\!\left(\frac{\mu(Y)\mu(Z)}{\Lambda}\right)
-
Y\!\left(\frac{\mu(Z)\mu(X)}{\Lambda}\right)
+
Z\!\left(\frac{\mu(X)\mu(Y)}{\Lambda}\right)
\\
+
\frac{1}{\Lambda}
\left[
-\mu([X,Y])\mu(Z)
+\mu([X,Z])\mu(Y)
+\mu([Y,Z])\mu(X)
\right].
\end{multline*}
\endgroup
We expand the right-hand side. Using the compatibility of the connection with the metric and the fact that $\mu(Y) = \riemannian{\mu^\sharp}{Y}$ for any $Y \in \Gamma(\tangentBundle)$, we obtain
\begingroup\small
\begin{align*}
X\!\left(\frac{\mu(Y)\mu(Z)}{\Lambda}\right)
&=
  \frac{\mu(Y)}{\Lambda}
  \bigl(
    \riemannian{\nabla_X \mu^\sharp}{Z} 
    + 
    \mu(\nabla_X Z)
  \bigr)
  \\
  &
  \quad
  +
  \frac{\mu(Z)}{\Lambda}
  \bigl(
    \riemannian{\nabla_X \mu^\sharp}{Y} 
    + 
    \mu(\nabla_X Y)
  \bigr)
  \\
  &\quad
  -
  \frac{2\mu(Y)\mu(X)}{\Lambda^2}\,\mu(\nabla_X \mu^\sharp),
\end{align*}
\endgroup
and similar relations for the terms involving $Y$ and $Z$. 
Using the fact that the connection $\nabla$ is torsion-free, \ie\ $[X, Y] = \nabla_X Y - \nabla_Y X$, and the previous relations in the Koszul formula, after some cancellations, we get
\begingroup\small
\begin{multline*}
  \hspace{-1em}
  2g(T_XY, Z)
  -
  \frac{2}{\Lambda}\,\mu(T_XY)\mu(Z)
  =
  \frac{\mu(Y)}{\Lambda}
  \bigl(
    \riemannian{\nabla_Z \mu^\sharp}{X}
    -
    \riemannian{\nabla_X \mu^\sharp}{Z}
  \bigr)
  \\
  \hspace{-0.5em}
  -
  \frac{\mu(Z)}{\Lambda}
  \bigl(
    \riemannian{\nabla_X \mu^\sharp}{Y}
    +
    \riemannian{\nabla_Y \mu^\sharp}{X}
  \bigr)
  +
  \frac{\mu(X)}{\Lambda}
  \bigl(
    \riemannian{\nabla_Z \mu^\sharp}{Y}
    -
    \riemannian{\nabla_Y \mu^\sharp}{Z}
  \bigr)
  \\
  \hspace{-0.5em}
  +
  \frac{2}{\Lambda}
  \Bigl[
    \mu(Y)\mu(Z)\mu(\nabla_X \mu^\sharp)
    +
    \mu(Z)\mu(X)\mu(\nabla_Y \mu^\sharp)
    -
    \mu(X)\mu(Y)\mu(\nabla_Z \mu^\sharp)
  \Bigr].
\end{multline*}
\endgroup
The goal now is to isolate a product of the form $\riemannian{\ast}{Z}$ on both sides of the previous equation. 
Using the $\cdot^\sharp$ isomorphism again and rearranging the terms yields
\begingroup\small
\begin{multline*}
  \hspace{-1em}
  2 \riemannian[auto]{
    T_XY - \frac{\mu(T_XY)}{\Lambda}\,\mu^\sharp
  }{Z}
  =
  \frac{1}{\Lambda}
  \riemannian[auto]{
    -\mu(Y)\nabla_X \mu^\sharp
    -\mu(X)\nabla_Y \mu^\sharp
  }{Z}
  \\
  -
  \frac{1}{\Lambda}
  \riemannian[auto]{ 
    \bigl[
      \riemannian{\nabla_X \mu^\sharp}{Y}
      +
      \riemannian{\nabla_Y \mu^\sharp}{X}
    \bigr]\mu^\sharp
  }{Z}
  \\
  +
  \frac{2}{\Lambda^2}
  \riemannian[auto]{
    \bigl[
      \mu(Y)\mu(\nabla_X \mu^\sharp)
      +
      \mu(X)\mu(\nabla_Y \mu^\sharp)
    \bigr]\mu^\sharp
  }{Z}
  \\
  +
  \frac{1}{\Lambda}
  \left[
    \mu(Y)\riemannian{\nabla_Z \mu^\sharp}{X}
    +
    \mu(X)\riemannian{\nabla_Z \mu^\sharp}{Y}
  \right]
  \\
  -
  \frac{2\mu(X)\mu(Y)}{\Lambda^2}\,
  \riemannian{\nabla_Z \mu^\sharp}{\mu^\sharp}
  .
\end{multline*}
\endgroup
We wish to manipulate the terms involving the covariant derivatives $\nabla_Z \mu^\sharp$. To this end, from metric compatibility, we see that
\begin{align*}
  \riemannian{\nabla_Z \mu^\sharp}{X}
  &
  =
  \riemannian{\nabla_X \mu^\sharp}{Z}
  +
  Z(\mu(X))
  -
  X(\mu(Z))
  +
  \mu([X, Z])
  \\
  &
  =
  \riemannian{\nabla_X \mu^\sharp}{Z}
  -
  \d\mu(X, Z),
\end{align*}
where the last equality is a consequence of Cartan's magic formula. Analogous relations hold for the other terms. 
Using these in the previous expression, we find
\begingroup\small
\begin{multline*}
  \hspace{-2em}
  2 \riemannian[auto]{
    T_XY - \frac{\mu(T_XY)}{\Lambda}\,\mu^\sharp
  }{Z}
  =
  \frac{1}{\Lambda}
  \riemannian[auto]{
    -\mu(Y)\nabla_X \mu^\sharp
    -\mu(X)\nabla_Y \mu^\sharp
  }{Z}
  \\
  -
  \frac{1}{\Lambda}
  \riemannian{
    \bigl[
      \riemannian{\nabla_X \mu^\sharp}{Y}
      +
      \riemannian{\nabla_Y \mu^\sharp}{X}
    \bigr]\mu^\sharp
  }{Z}
  \\
  +
  \frac{2}{\Lambda^2}
  \riemannian[auto]{
    \bigl[
      \mu(Y)\mu(\nabla_X \mu^\sharp)
      +
      \mu(X)\mu(\nabla_Y \mu^\sharp)
    \bigr]\mu^\sharp
  }{Z}
  \\
  \hspace{-2em}
  +
  \frac{1}{\Lambda}
  \left[
    \mu(Y)\bigl(\riemannian{\nabla_X \mu^\sharp}{Z} - \d\mu(X,Z)\bigr)
    +
    \mu(X)\bigl(\riemannian{\nabla_Y \mu^\sharp}{Z} - \d\mu(Y,Z)\bigr)
  \right]
  \\
  -
  \frac{2\mu(X)\mu(Y)}{\Lambda^2}
  \bigl(
    \riemannian{\nabla_{\mu^\sharp}\mu^\sharp}{Z}
    -
    \d\mu(\mu^\sharp,Z)
  \bigr)
  \\
  =
  \frac{2}{\Lambda}
  \riemannian[auto]{
    \left[
      \frac{1}{\Lambda}
      \riemannian{
        \mu(Y)\nabla_X \mu^\sharp + \mu(X)\nabla_Y \mu^\sharp
      }{\mu^\sharp}
      -
      \riemannian{\nabla_X \mu^\sharp}{Y}
    \right]\mu^\sharp
  }{Z}
  \\
  +
  \frac{1}{\Lambda}
  \riemannian[auto]{
    \d\mu(X,Y)
    \mu^\sharp
  }{Z}
  -
  \frac{2 \mu(X)\mu(Y)}{\Lambda^2}
  \riemannian[auto]{
    \nabla_{\mu^\sharp}\mu^\sharp
  }{Z}
  \\
  +
  \frac{1}{\Lambda}\,
  \d\mu\!\left(
    \frac{2\mu(X)\mu(Y)}{\Lambda}\,\mu^\sharp
    - \mu(Y)X - \mu(X)Y,
    Z
  \right),
\end{multline*}
\endgroup
where the last equality follows from the linearity of $\d\mu$ in its first component. 
Using the musical isomorphism $\cdot^\sharp$ on $\d\mu(\ast, \cdot)$ now gives
\begingroup\small
\begin{multline*}
  T_XY - \frac{\mu(T_XY)}{\Lambda}\,\mu^\sharp
  =
  \frac{1}{\Lambda^2}
  \riemannian{
    \mu(Y)\nabla_X \mu^\sharp + \mu(X)\nabla_Y \mu^\sharp
  }{\mu^\sharp}
  \mu^\sharp
  \\
  +
  \frac{1}{\Lambda}
  \left[
    -
    \riemannian{
      \nabla_X \mu^\sharp
    }{Y}
    +
    \frac{1}{2}\,\d\mu(X,Y)
  \right]\mu^\sharp
  \\
  -
  \frac{\mu(X)\mu(Y)}{\Lambda^2}\,\nabla_{\mu^\sharp}\mu^\sharp
  \\
  +
  \frac{1}{2\Lambda}\,
  \d\mu\!\left(
    \frac{2\mu(X)\mu(Y)}{\Lambda}\,\mu^\sharp
    - \mu(Y)X - \mu(X)Y,
    \cdot
  \right)^\sharp.
\end{multline*}
\endgroup
By decomposing $T_XY$ in components parallel and orthogonal to $\mu^\sharp$, and defining
\[
\mu^\sharp_\perp
\coloneqq
\frac{\mu(X)\mu(Y)}{\Lambda}\,\mu^\sharp
-
\frac{\mu(Y)}{2}\,X
-
\frac{\mu(X)}{2}\,Y,
\]
it is straightforward to establish
\begingroup\small
\begin{equation}
\label{eq:difference-tensor}
\begin{aligned}
  T_XY
  ={}&
  \left[
    \frac{1}{\Lambda}
    \riemannian{
      \mu(Y)\nabla_X \mu^\sharp + \mu(X)\nabla_Y \mu^\sharp
    }{\mu^\sharp}
    -
    \riemannian{\nabla_X \mu^\sharp}{Y}
  \right]\mu^\sharp
  \\
  &
  \hspace{-3em}
  +
  \left[
    \frac{1}{2}\,\d\mu(X,Y)
    -\frac{\mu(X)\mu(Y)}{\Lambda}
    \frac{\mu(\nabla_{\mu^\sharp}\mu^\sharp)}{\riemanniannorm{\mu^\sharp}^2}
    +
    \frac{\Lambda\,\d\mu(\mu^\sharp_\perp,\mu^\sharp)}{\riemanniannorm{\mu^\sharp}^2}
  \right]\mu^\sharp
  \\
  &-
  \frac{\mu(X)\mu(Y)}{\Lambda^2}
  \left[
    \nabla_{\mu^\sharp}\mu^\sharp
    -
    \frac{\mu(\nabla_{\mu^\sharp}\mu^\sharp)}{\riemanniannorm{\mu^\sharp}^2}\,\mu^\sharp
  \right]
  \\
  &+
  \frac{1}{\Lambda}\,\d\mu(\mu^\sharp_\perp,\cdot)^\sharp
  -
  \frac{\d\mu(\mu^\sharp_\perp,\mu^\sharp)}{\riemanniannorm{\mu^\sharp}^2}\,\mu^\sharp,
\end{aligned}
\end{equation}
\endgroup
which concludes the proof.
\end{proof}

\section{Riemannian Trajectory Tracking}
\subsection{Metric preconditioning for connection control systems}

Let $(\mathcal M,g)$ be a smooth Riemannian manifold with Levi--Civita connection $\nabla$. An Affine Connection Control System on $\mathcal M$ with control distribution $\mathcal{D} \subset \tangentBundle$ is of the form
\begin{equation}
    \label{eq:connection-control-affine-M}
    \nabla_{\dot q}\dot q
    =
    Y_0(q(t))+
    \sum_{j=1}^k u^j(t)Y_j(q(t)),
\end{equation}
where $Y_0\in\mathfrak X(\mathcal M)$ is the drift, $Y_1,\ldots,Y_k\in\mathfrak X(\mathcal M)$ is a local frame spanning $\mathcal{D}$, and $u^j\in\mathcal C^0([0, t_{\text{f}}])$ are the control inputs with respect to this local frame \cite{Lewis}, for some final time $t_{\text{f}} > 0$.
In the case that $\mathcal D = \tangentBundle$, we say that the system is fully-actuated, and so the right-hand side of \eqref{eq:connection-control-affine-M} may be chosen arbitrarily in $\tangentSpace{q}$. In the underactuated case, $\mathcal D$ is a proper subbundle of $\tangentBundle$, and only the projection of a desired covariant acceleration onto the control distribution can generally be realized.

Let $q_d\in \mathcal C^2([0,t_{\text{f}}],\mathcal M)$ be a reference trajectory. Our aim is to design control inputs $u^j(t)$ such that the closed-loop integral curves of \eqref{eq:connection-control-affine-M} track $q_d(t)$ on $[0,t_{\text{f}}]$, at which point we say that the control system \textit{tracks} $q_d$. To properly study this convergence, we require a notion of \textit{tracking error}. Let $\retractionSymbol: \tangentBundle \to \mathcal{M}$ denote a retraction on $\mathcal{M}$, and assume that the inverse retraction
\(
\retractionSymbol_q^{-1}(q_d(t))\in \tangentSpace{q}
\)
is well-defined along the portion of the state space under consideration, and let $\mathrm{VT}^{\mathrm{retr}} \colon \tangentBundle \to \tangentBundle$ denote the vector transport associated to $\mathrm{retr}$. We define the configuration errors
\[
e_q \coloneq -\inverseRetract{q}{q_d}\in \tangentSpace{q},
\qquad
e_v \coloneq \dot q - \mathrm{VT}_{q_d \to q}^{\mathrm{retr}} \dot q_d\in \tangentSpace{q},
\]
which, within a tubular neighborhood of the nominal trajectory $q_d$, satisfy the properties that $e_q = e_v = 0$ if and only if $q \equiv q_d$. In the special case where the retraction is the exponential map, and vector transport is parallel transport, we obtain configuration errors consistent with those studied in \cite{configError}.
Consider the PD control with nominal trajectory feedback given by:
\begin{equation}
    \label{eq:old-control-rewrite}
    u_{\mathrm{f}}
    \coloneq 
    \mathrm{VT}_{q_d \to q}^{\mathrm{retr}}
    \left(
        \nabla_{\dot q_d}\dot q_d
    \right)
    -
    k_q e_q - k_v e_v.
\end{equation}
In the case that the system \eqref{eq:connection-control-affine-M} is fully-actuated, the drift term can be removed via feedback, and the closed-loop dynamics of the corresponding driftless system with the above control has the same first-order form as the Euclidean PD tracking dynamics:
\[
    \dot e_q
    =
    e_v
    +
    O(\|e_q\|),
    \qquad
    \dot e_v
    =
    -k_qe_q-k_ve_v
    +
    O(\|e_q\|+\|e_v\|).
\]
where the higher-order terms depend on the curvature of the connection, the chosen retraction, and the chosen vector transport. Hence, the gains $k_q,\, k_v \ge 0$ can be interpreted as parameters which shape the eigenvalues of the local linearized error dynamics. The goal is therefore to choose gains such that the tracking errors remain small on $[0, t_{\text{f}}]$, and converge asymptotically to zero when the reference trajectory is extended in time.

We now explain how the metric modification can be used as a \emph{geometric preconditioner} for this feedback law. Fix a vector field $X\in\mathfrak X(\mathcal M)$, choose $\lambda\ge 0$, and set
\(
\mu\coloneq \lambda X^\flat,
\)
and
\begin{equation}
\label{eq:g-lambda}   
 g_\lambda
 \coloneq
 g-
 \frac{\mu\otimes\mu}{1+\|\mu^\sharp\|_g^2}
    =
 g-
 \frac{\lambda^2 X^\flat\otimes X^\flat}{1+\lambda^2\|X\|_g^2}.
\end{equation}
Denote by $\widetilde\nabla^\lambda$ the Levi--Civita connection of $g_\lambda$, and write
\[
\widetilde\nabla^\lambda_YZ
=
\nabla_YZ
+
T^\lambda_Y Z,
\]
where $T^\lambda$ is the difference tensor \eqref{eq:difference-tensor} from Lemma~\ref{lemma:new-levi-civita}.

The role of $X$ is to encode a preferred direction of motion. In particular, if $Y$ is $g$-orthogonal to $X$, then
\[
\|Y\|_{g_\lambda}=\|Y\|_g,
\]
whereas along $X$ one has
\begin{equation}
    \label{eq:modified_norm_lambda_rewrite}
    \|X\|_g^2 \ge \|X\|_{g_\lambda}^2
    =
    \frac{\|X\|_g^2}{1+\lambda^2\|X\|_g^2}
    \to 0
    \quad\text{as } \lambda\to\infty .
\end{equation}
Thus, motion tangent to $X$ becomes progressively cheaper as $\lambda$ increases, while directions $g$-orthogonal to $X$ are unchanged. If $X$ is chosen so that $X(q_d(t))$ approximates $\dot q_d(t)$ along the reference trajectory, then $g_\lambda$ biases the kinetic geometry of the system toward the nominal motion.

This observation motivates the following control design. Rather than prescribing the closed-loop dynamics relative to the original connection $\nabla$, we prescribe the closed-loop dynamics relative to the modified connection:
\begin{equation}
    \label{eq:modified_connection_target}
    \widetilde\nabla^\lambda_{\dot q}\dot q
    =
    \mathrm{VT}_{q_d \to q}^{\mathrm{retr}}
    \left(
        \widetilde\nabla^\lambda_{\dot q_d}\dot q_d
    \right)
    -
    k_q e_q - k_v e_v.
\end{equation}
Using $\widetilde\nabla^\lambda=\nabla+T^\lambda$, \eqref{eq:modified_connection_target} is equivalent to
\[
\nabla_{\dot q}\dot q
=
\mathrm{VT}_{q_d \to q}^{\mathrm{retr}}
\left(
    \nabla_{\dot q_d}\dot q_d
    +
    T^\lambda_{\dot q_d}\dot q_d
\right)
-
k_q e_q
-
k_v e_v
-
T^\lambda_{\dot q}\dot q.
\]
Equivalently, in terms of the nominal feedforward law \eqref{eq:old-control-rewrite}, this yields the \emph{preconditioned} fully actuated control law
\begin{equation}
    \label{eq:preconditioned-control-rewrite}
    u_{\mathrm{pre}}
    \coloneq
    u_{\mathrm{f}}
    +
    \mathrm{VT}_{q_d \to q}^{\mathrm{retr}}
    \left(
        T^\lambda_{\dot q_d}\dot q_d
    \right)
    -
    T^\lambda_{\dot q}\dot q.
\end{equation}
Accordingly, if the system is fully actuated, \ie if $\mathcal D = \tangentBundle$, then \eqref{eq:preconditioned-control-rewrite} realizes exactly the closed-loop dynamics \eqref{eq:modified_connection_target} within the affine connection control system \eqref{eq:connection-control-affine-M}. If the system is underactuated, a natural surrogate is to project $u_{\mathrm{pre}}$ orthogonally onto the control distribution:
\[
u_{\mathrm{pre}}^{\mathcal D}
\coloneq
\Pi_{\mathcal D}\!\big(u_{\mathrm{pre}}\big),
\]
where $\Pi_{\mathcal D}$ denotes the $g$-orthogonal projection onto $\mathcal D\subset \tangentBundle$.

We remark that it is not claimed here that the $g_\lambda$-geodesics converge to integral curves of $X$ as $\lambda\to\infty$. Establishing such a result would require a separate analysis of the geodesic spray, or equivalently of the associated energy functionals, and is beyond the scope of the present work. Our aim here is more modest: to propose a geometrically natural preconditioning strategy, to identify its exact implementation in the closed-loop equations, and to provide heuristic justification for why it should improve tracking performance when $X$ is well aligned with the desired direction of motion.

\begin{remark}
    \label{rm:preconditioner-gains-linearization}
    In the special case where $\mathcal{M} = \mathbb{R}$ and $g$ is the standard Euclidean inner product, the control system \eqref{eq:modified_connection_target} takes the form
    \[
        \ddot{x}
        =
        \frac{\lambda^2 X(x) X'(x)}
        {1 + \lambda^2 X(x)^2}\dot{x}^2
        +
        u_{\mathrm{f}} .
    \]
    Letting $v = \dot{x}$, and defining the configuration errors
    $e_x = x - x_d$ and $e_v = v - v_d$, set
    \[
        F_\lambda(x)
        \coloneq
        \frac{\lambda^2 X(x) X'(x)}
        {1 + \lambda^2 X(x)^2}.
    \]
    To make the zero-error trajectory an exact solution, we choose the
    feedforward term so that the modified covariant acceleration of the
    desired trajectory is matched at zero error:
    \[
        u_{\mathrm f}
        =
        \ddot{x}_d
        -
        F_\lambda(x_d)v_d^2
        +
        u_{\mathrm{pd}},
        \qquad
        u_{\mathrm{pd}}
        =
        -k_xe_x-k_ve_v .
    \]
    We then obtain the error dynamics
    \begin{align*}
        \dot{e}_x &= e_v,\\
        \dot{e}_v
        &=
        F_\lambda(e_x+x_d)(e_v+v_d)^2
        -
        F_\lambda(x_d)v_d^2
        +
        u_{\mathrm{pd}} .
    \end{align*}
    In particular, $(e_x,e_v)=(0,0)$ is an exact solution of the error
    dynamics. Taylor expanding $F_\lambda$ around $e_x=0$ and truncating
    higher-order terms, we obtain the following linearized closed-loop
    error dynamics:
    \[
        \begin{bmatrix}
            \dot{e}_x \\
            \dot{e}_v
        \end{bmatrix}
        =
        \begin{bmatrix}
            0 & 1 \\
            F_\lambda'(x_d)v_d^2 - k_x
            &
            2F_\lambda(x_d)v_d - k_v
        \end{bmatrix}
        \begin{bmatrix}
            e_x \\ e_v
        \end{bmatrix}.
    \]
    Thus, at the level of the frozen-time linearization, the
    preconditioner modifies the effective damping and stiffness according
    to
    \[
        k_v \mapsto k_v - 2F_\lambda(x_d)v_d,
        \qquad
        k_x \mapsto k_x - F_\lambda'(x_d)v_d^2 .
    \]
    The corresponding Routh--Hurwitz conditions for the frozen-time
    linearization are
    \[
        2F_\lambda(x_d)v_d < k_v,
        \qquad
        F_\lambda'(x_d)v_d^2 < k_x .
    \]
    These inequalities hold for all positive gains $k_x,k_v$ in the
    special case where
    \[
        F_\lambda(x_d)v_d < 0,
        \qquad
        F_\lambda'(x_d) < 0 .
    \]
    If the nominal trajectory is an integral curve of $X$, so that
    $v_d = X(x_d)$, then the first condition is equivalent to
    $X'(x_d)<0$, assuming $X(x_d)\neq 0$. Hence, the damping contribution
    is favorable when the vector field is contractive near the nominal
    trajectory. On the other hand, when
    \[
        F_\lambda(x_d)v_d > 0
        \qquad\text{or}\qquad
        F_\lambda'(x_d) > 0,
    \]
    the preconditioner acts against the ordinary Euclidean PD controller
    by reducing the effective damping or stiffness, and larger gains may
    be required to stabilize the linearized error dynamics.

    We emphasize that this calculation is only a first-order, frozen-time
    diagnostic. Since $x_d(t)$ is generally time-dependent, the full
    variational equation is time-varying, and pointwise stability of the
    frozen-time matrices does not by itself imply nonlinear asymptotic
    stability. Nevertheless, the calculation illustrates that the effect
    of the preconditioner is highly sensitive to the local behavior of
    the vector field used to modify the geometry around the nominal
    trajectory. We leave a deeper analysis of these conditions to future
    work.
\end{remark}

\subsection{Rigid Body Control on $\SE(3)$}
Consider the Lie group
\[
\SE(3)=
\left\{
\begin{bmatrix}
    R & p \\ 0 & 1
\end{bmatrix}
\in\mathbb R^{4\times 4}
\,\middle\vert\,
R\in\SO(3),\ p\in\mathbb R^3
\right\}.
\]
We frequently compress the notation to $(R,p)\in\SO(3)\times\mathbb R^3$, with the group operation
\[
(R_1,p_1)\cdot(R_2,p_2)
=
(R_1R_2,R_1p_2+p_1).
\]
The Lie algebra is given by
\[
\se(3)=
\left\{
\begin{bmatrix}
    \hat\Omega & v \\ 0 & 0
\end{bmatrix}
\,\middle\vert\,
\hat\Omega\in\so(3),\ v\in\mathbb R^3
\right\}.
\]
Using the hat isomorphism $\hat{\cdot}\colon\mathbb R^3\to\so(3)$ defined implicitly by $\hat x y\coloneq x\times y$ for all $x,y\in\mathbb R^3$, we identify $\se(3)\cong\mathbb R^3\times\mathbb R^3$ via $(\Omega,v)\longmapsto (\hat\Omega,v)$.
We equip $\mathbb R^3\times\mathbb R^3$ with the inner product
\[
\langle(\Omega_1,v_1),(\Omega_2,v_2)\rangle
=
\Omega_1^T\mathbb J\Omega_2+v_1^Tv_2,
\]
where $\mathbb J\in\operatorname{Sym}_3^+(\mathbb R)$ is a symmetric, positive-definite \emph{inertia tensor}. This inner product is pulled back to $\se(3)$ through the hat isomorphism and then left-translated to obtain a left-invariant Riemannian metric on $\SE(3)$ used in the simulation.

For a curve $(R(t),p(t))\in\SE(3)$, we use the left-trivialized velocity $(\Omega,v)\in\mathbb R^3\times\mathbb R^3$ defined by
\[
\dot R=R\hat\Omega,
\qquad
\dot p=Rv.
\]
With respect to the above metric, the unforced rigid-body terms take the form
\[
\dot v=-\Omega\times v,
\qquad
\dot\Omega=\mathbb J^{-1}(\mathbb J\Omega\times\Omega).
\]
In many control problems on $\SE(3)$, such as trajectory tracking of quadrotor UAVs, the rotational dynamics are fully actuated while the translational dynamics are underactuated \cite{Lee}. We therefore take the control distribution to be generated by the $\se(3)$ elements $(\hat e_1,0),\quad(\hat e_2,0),\quad(\hat e_3,0),\quad(0,e_3)$,
where $\{e_1,e_2,e_3\}$ is the standard basis of $\mathbb R^3$. A gravitational drift of the form $\rho e_3$ in the translational equation may be removed by the feedback $u^4(t)=f(t)-\rho$. Since this force lies in the control distribution, we proceed under the driftless form without loss of generality:
\begin{align*}
    \dot R &= R\hat\Omega,
    &
    \dot p &= Rv,\\
    \dot v &= -\Omega\times v+f e_3,
    &
    \dot\Omega &= \mathbb J^{-1}(\mathbb J\Omega\times\Omega)+ \tau,
\end{align*}
with $f\in\mathbb R$ the thrust and $\tau\in\mathbb R^3$ the body torque input.

Since the system is underactuated, it will not generally be possible to track arbitrary nominal trajectories $(R_d,p_d)$ in both position and orientation. For our purposes, it is more pressing to track a given nominal trajectory in position, which allows us to use a strategy similar to that proposed in \cite{Lee}. Observe that the translational dynamics can equivalently be written as
\[
\ddot p=fRe_3.
\]
We treat this system as a fully actuated double integrator $\ddot p=\nu$, with $\nu$ designed as in equation \eqref{eq:old-control-rewrite} or \eqref{eq:preconditioned-control-rewrite}. We then choose the desired orientation $R_d$ so that
\begin{equation}
  \label{eq:desired-orientation}
  fR_de_3=\nu
  \qquad\implies\qquad
  R_de_3=\frac{\nu}{\|\nu\|}.
\end{equation}
In other words, the nominal attitude trajectory $R_d$ is chosen so that the translational control input $fRe_3$ tracks the desired acceleration $\nu$. Given $\nu$, this uniquely determines the final column $b_3=R_de_3$ of $R_d$. To specify the remaining degrees of freedom, we additionally specify the first column $b_1=R_de_1$, which must satisfy $\|b_1\|=1$ and $b_1\cdot b_3=0$. We choose
\[
b_1=
\frac{e_1\times b_3}{\|e_1\times b_3\|},
\]
which is valid whenever $b_3\ne\pm e_1$. We note that there is no choice in $b_1$ which varies continuously with $b_3$ while satisfying $\|b_1\|=1$ and $b_1\cdot b_3=0$, as this would be equivalent to the existence of a continuous, nowhere vanishing vector field on the sphere $\mathbb S^2$, which contradicts the Hairy Ball Theorem.

\section{Numerical Experiments}
We consider a maze navigation problem on $\SE(3)$ with three straight corridors connected together in series. Let $P[p,q]$ denote the plane in $\R^3$ orthogonal to the $xy$-plane and running through the points $p \in \R^2$ and $q \in \R^2$ on the $xy$-plane. The maze is given by the region in $\R^3$ bounded by $0 \le z \le 2$, and the planes $P[(0,\,-2),\,(0,\,0)]$, $P[(1,\,-2),\,(1,\,1/2)]$, 
$P[(1,\,1/2),\,(-1/2,\,1/2)]$, $P[(-1/2,\,1/2),\,(-1/2,\,3/2)]$, $P[(0,\,0),\,(-1,\,0)]$, $P[(-1, \,0), \,(-1, \,3/2)]$.

To test the performance of our method, we compare it to the one presented in \cite{Lee}, referred to as ``Euclidean'' in the following. 
We do so by conducting two experiments.
The inertia tensor is taken from \cite{Pounds}: $\bbJ \coloneq \text{diag}([0.082, 0.0845, 0.1377])$, and it is fixed throughout our numerics.
\footnote{The code can be found at \url{https://doi.org/10.5281/zenodo.20442932}.} In all cases, the torque input $\tau$ is chosen as in \cite{Lee}.

\subsection{Tracking an Integral Curve}
We modify the translational geometry only.
That is, we design a vector field $X \in \Gamma(\bbR^3)$ with bounded norm and low divergence that flows through the given maze, \ie its integral curves follow the center line of the corridors that form the maze.
We impose a non-constant tangential speed on $X$ and a small wave in the $z$-component, so that the acceleration of the reference trajectory does not vanish along ``straight'' stretches of the curve.
In practice, this is achieved by combining three smooth vector fields for each of the three corridors, each equipped with local coordinates $(\rho_i, s_i)$ describing the transverse offset and the arc-length along a corridor, increasing from the starting point $p_0 \coloneq [1/2, 18/10, 1]^\top$ in the direction of the endpoint of the desired trajectory, $p_{t_{\text{f}}} \coloneq [- 3/4, 5/4, 1]^\top$.
We use local vector fields 
\begin{equation*}
  F_i(\rho_i, s_i)
  \coloneq
  v(s_i) \hat t_i
  -
  \gamma \rho_i \hat n_i
  ,
\end{equation*}
where $v(s)$ % = v_0 (1 + \alpha \tanh \frac{s - \bar s}{l})$ 
is a mildly varying tangential speed, $\gamma > 0$ is a transverse centering gain that makes nearby integral curves stay within the corridor, and $\hat t_i, \hat n_i$ are the unit tangent and normal vectors of the $i$-th corridor, $i=1,2,3$, respectively. 
Each of these vector fields is activated by a smooth compactly-supported weight $w_i \coloneq \exp\big(- \frac{\rho_i^2}{\sigma^2}\big) \, l(s_i, s_{\text{min}, i}, s_{\text{max},i })$, where $l$ is a logistic gate.
A wave $W = A \sin(\omega \phi + \psi)$, with an amplitude $A = 0.035$, a phase $\phi = \frac{\sum_{i=1}^{3} w_i \phi_i}{\sum_{i=1}^{3}w_i}$, and shift $\psi = \pi/4$ is added to the weighted sum of these vector fields to prevent zero material acceleration along straight segments, which would result in poor numerical conditioning. 
We then saturate the resulting vector field with a prescribed magnitude to obtain $X$.
Finally, we construct an integral curve of $X$ by accurately solving its induced ODE numerically with initial value given by $p_0$, and utilize this as the desired trajectory $p_d$.

We extend $X$ to a vector field on $\SE(3)$ with a $3 \times 3$ matrix of zeros.
As per Remark \ref{rm:modification-via-vector}, we set $\mu = X^\flat$.
This choice of $\mu$ also enables the computation of the difference tensor only for the evolution on $\bbR^3$, rather than having to compute it for the full Levi--Civita connection of $\SE(3)$.
Then, we obtain the control $\nu = u_{\text{f}} - T^\lambda_{\dot p} e_v$, and construct $R_d$ according to \eqref{eq:desired-orientation}.

To select gains for both methodologies, we run a sweep over a grid of gains for the rotational and translational components, as well as for the velocity and attitude.
The objective of the sweep is chosen so as to minimize 
$$\|u\|^2_{L^2} + \frac{1}{4} \max \|u\|^2$$
over all of the gains that successfully solve the tracking problem in the given grid.
We saturate the thrust and body torque inputs with $f_\text{max} = 10$, and $\tau_\text{max} = 100$, respectively.
The starting point for all our simulations is $q_0 \coloneq p_0 + [1/4, -5/100, 0]^\top$.
In our case, the procedure returned $k_R^E = 5.75, k_q^E = 5.75, k_v^E = 7.525, k_\Omega^E = 2.575$ and $k_R^P = 5.75, k_q^P = 5.75, k_v^P = 5.05, k_\Omega^P = 5.05$ for the Euclidean and preconditioned strategies, respectively.
We observe only a higher rotation gain for our preconditioner.
The results of the simulation with this choice of gains are summarized in Figure \ref{fig:exact-10-100-integral-curve}.
The preconditioned strategy yields lower position tracking errors for $2/3$ of the simulation time. 
We remark that the position errors being higher in the first third is due to the parametrization of the trajectories.
Indeed, the Euclidean trajectory starts out faster and reaches a point that is closer to the nominal trajectory in the first third of the simulation. But when the preconditioned trajectory catches up, it does a significantly better job at tracking the desired curve.
We refer the reader to the animation in the linked repository.
Crucially, the preconditioner is able to track the desired trajectory visibly better from slightly before the first turn onwards.
A similar comment holds regarding the velocity tracking errors.
The control are mostly comparable, with exceptions around the elbows of the maze, where the body torque input is stronger for the preconditioned PD.
We also note that the thrust is halved for the preconditioned PD around the second corridor.  

This suggests that a possible future approach is to design a hybrid control scheme, where ordinary PD control is used initially until the tracking errors are sufficiently small, and then the preconditioner is activated to keep tracking errors low throughout high-jerk motions that the Euclidean PD is unable to account for without large controller gains.
\begin{figure}[t]
  \centering
  %--- (a) Trajectory: full column width ---
  \subcaptionbox{Trajectory, $xy$ projection. Orange: $p_0$, green: $p_{t_{\text{f}}}$. \label{fig:traj}}{%
    \begin{tikzpicture}
      \begin{axis}[
        legend style={font=\scriptsize},
        legend pos=north east,
        width=0.6\columnwidth, height=0.6\columnwidth%4.75cm
      ]
        \foreach \xa/\ya/\xb/\yb in {
          0.0/-2.0/0.0/0.0, 1.0/-2.0/1.0/0.5,
          1.0/0.5/-0.5/0.5, 0.0/0.0/-1.0/0.0,
          -1.0/0.0/-1.0/1.5, -0.5/0.5/-0.5/1.5
        }{ \addplot[black, thick] coordinates {(\xa,\ya) (\xb,\yb)}; }
        
        \addplot[black, dashed, ultra thick]
          table[col sep=comma, x=q_ref_x, y=q_ref_y]{data/trajectory.csv};
        \addlegendentry{Desired}
        
        \addplot[only marks, mark options={draw opacity=0, fill=Orange, mark size=1.5pt}]
          coordinates {(0.5,-1.8)};
          
        \addplot[only marks, mark options={draw opacity=0, fill=Green, mark size=1.5pt}]
          coordinates {(-0.75,1.25)};
          
        \addplot[red, thick]
          table[col sep=comma, x=q_x, y=q_y]{data/exact_f10_u100_euclidean_trajectory.csv};
        % \addlegendentry{Euclidean PD}
        
        \addplot[blue, thick]
          table[col sep=comma, x=q_x, y=q_y]{data/exact_f10_u100_preconditioned_trajectory.csv};
        % \addlegendentry{Preconditioned PD}
        
        \addplot[blue!60, quiver={u=\thisrow{u}, v=\thisrow{v}, scale arrows=0.1}, -stealth]
          table[col sep=comma, x=x, y=y]{data/vector_field.csv};
          
      \end{axis}
    \end{tikzpicture}
  }
  \\[1em]
  %--- (b) Position error ---
  \subcaptionbox{Position error $\|p-p_d\|$.\label{fig:pos-err}}{%
    \begin{tikzpicture}
      \begin{axis}[
        legend style={font=\scriptsize},
        yticklabel style={font=\small},
        xticklabel style={font=\small},
        width=0.55\columnwidth, height=4.75cm
      ]
        \addplot[red, thick]
          table[col sep=comma, x=t, y=pos_err]{data/exact_f10_u100_euclidean_errors_control.csv};
        \addlegendentry{Euclidean}
        
        \addplot[blue, thick]
          table[col sep=comma, x=t, y=pos_err]{data/exact_f10_u100_preconditioned_errors_control.csv};
        \addlegendentry{Precond.}
      \end{axis}
    \end{tikzpicture}
  }%\hfill
  % --- (c) Velocity error ---
  \subcaptionbox{Velocity error $\|Rv-v_d\|$.\label{fig:vel-err}}{%
    \begin{tikzpicture}
      \begin{axis}[
        yticklabel style={font=\small},
        xticklabel style={font=\small},
        width=0.55\columnwidth, height=4.75cm
      ]
        \addplot[red, thick]
          table[col sep=comma, x=t, y=vel_err]{data/exact_f10_u100_euclidean_errors_control.csv};
        \addplot[blue, thick]
          table[col sep=comma, x=t, y=vel_err]{data/exact_f10_u100_preconditioned_errors_control.csv};
      \end{axis}
    \end{tikzpicture}
  }%\hfill
  \\[1em]
  %--- (d) Control magnitude ---
  \subcaptionbox{Body torque $\|\tau\|$.\label{fig:ctrl}}{%
    \begin{tikzpicture}
      \begin{axis}[
        legend pos=north east,
        legend style={font=\small},
        yticklabel style={font=\small},
        xticklabel style={font=\small},
        width=0.55\columnwidth, height=4.75cm
      ]
        \addplot[red, thick]
          table[col sep=comma, x=t, y=u_norm]{data/exact_f10_u100_euclidean_errors_control.csv};
        
        \addplot[blue, thick]
          table[col sep=comma, x=t, y=u_norm]{data/exact_f10_u100_preconditioned_errors_control.csv};
        
      \end{axis}
    \end{tikzpicture}
  }
  % \\[1em]
  %--- (f) Thrust ---
  \subcaptionbox{Thrust $f$.\label{fig:ctrl}}{%
    \begin{tikzpicture}
      \begin{axis}[
        legend pos=north east,
        legend style={font=\small},
        yticklabel style={font=\small},
        xticklabel style={font=\small},
        width=0.55\columnwidth, height=4.75cm
      ]
        \addplot[red, thick]
          table[col sep=comma, x=t, y=f]{data/exact_f10_u100_euclidean_errors_control.csv};
        
        \addplot[blue, thick]
          table[col sep=comma, x=t, y=f]{data/exact_f10_u100_preconditioned_errors_control.csv};
        
      \end{axis}
    \end{tikzpicture}
  }
  \caption{Tracking of an integral curve of $X$. Errors and controls plotted as functions of time.}
  \label{fig:exact-10-100-integral-curve}
\end{figure}
\subsection{Tracking a Different Curve}
In this experiment, we keep $\mu$ as in the previous section, but we change the desired trajectory. 
We generate a different trajectory by integrating another vector field that navigates the maze in a similar fashion to the previous experiment, but with different parameters.
The result is a trajectory $\tilde p_d$ that, while still navigating the given maze, is \emph{not} an integral curve of the vector field $X$ used to modify the original metric.
The two curves have a mean distance of around $0.3$ in both $x$ and $y$ components.
In particular, the two trajectories differ the most along the second corridor of the maze, including their inflection points.
That is, this new trajectory has a larger jerk and snap than the original one.

We kept the same control saturation from the previous experiment.
In this case, the gain sweep returned $k_R^E = 15.25, k_q^E = 5.75, k_v^E = 7.525, k_\Omega^E = 5.05$ and $k_R^P = 20.0, k_q^P = 5.75, k_v^P = 10.0, k_\Omega^P = 7.525$ for the Euclidean and preconditioned strategies, respectively.
The results of this simulation are summarized in Figure \ref{fig:inexact-10-100-integral-curve}.
The preconditioned strategy yields lower position tracking errors for around $75\%$ of the simulation time, but needs higher gains for both the rotational component, as well as for the velocity.
We observe here that the Euclidean PD is more accurate along the second corridor of the maze, in contrast to the previous case, which is in line with Remark \ref{rm:preconditioner-gains-linearization}.
The velocity tracking errors are also smaller for the preconditioner for about $75\%$ of the simulation time, with the Euclidean strategy tracking the velocity of the curve better around its turning points.
We refer the reader to the animation in the linked repository.
The controls are largely comparable along the whole trajectory, with the highest spikes being somewhat delayed for the body torque.
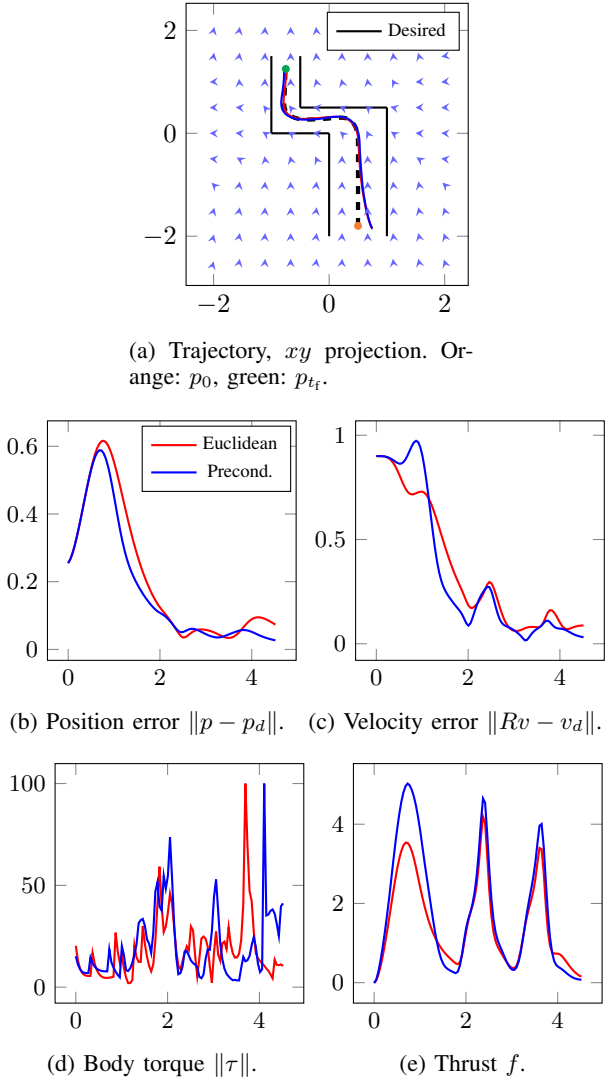
\begin{figure}[t]
  \centering
  %--- (a) Trajectory: full column width ---
  \subcaptionbox{Trajectory, $xy$ projection. Orange: $p_0$, green: $p_{t_{\text{f}}}$. \label{fig:traj}}{%
    \begin{tikzpicture}
      \begin{axis}[
        legend style={font=\scriptsize},
        legend pos=north east,
        width=0.6\columnwidth, height=0.6\columnwidth%4.75cm
      ]
        \foreach \xa/\ya/\xb/\yb in {
          0.0/-2.0/0.0/0.0, 1.0/-2.0/1.0/0.5,
          1.0/0.5/-0.5/0.5, 0.0/0.0/-1.0/0.0,
          -1.0/0.0/-1.0/1.5, -0.5/0.5/-0.5/1.5
        }{ \addplot[black, thick] coordinates {(\xa,\ya) (\xb,\yb)}; }
        
        \addplot[black, dashed, ultra thick]
          table[col sep=comma, x=q_ref_x, y=q_ref_y]{data/inexact_trajectory.csv};
        \addlegendentry{Desired}
        
        \addplot[only marks, mark options={draw opacity=0, fill=Orange, mark size=1.5pt}]
          coordinates {(0.5,-1.8)};
          
        \addplot[only marks, mark options={draw opacity=0, fill=Green, mark size=1.5pt}]
          coordinates {(-0.75,1.25)};
          
        \addplot[red, thick]
          table[col sep=comma, x=q_x, y=q_y]{data/inexact_f10_u100_euclidean_trajectory.csv};
        % \addlegendentry{Euclidean PD}
        
        \addplot[blue, thick]
          table[col sep=comma, x=q_x, y=q_y]{data/inexact_f10_u100_preconditioned_trajectory.csv};
        % \addlegendentry{Preconditioned PD}
        
        \addplot[blue!60, quiver={u=\thisrow{u}, v=\thisrow{v}, scale arrows=0.1}, -stealth]
          table[col sep=comma, x=x, y=y]{data/vector_field.csv};
          
      \end{axis}
    \end{tikzpicture}
  }
  \\[1em]
  %--- (b) Position error ---
  \subcaptionbox{Position error $\|p-p_d\|$.\label{fig:pos-err}}{%
    \begin{tikzpicture}
      \begin{axis}[
        legend style={font=\scriptsize},
        yticklabel style={font=\small},
        xticklabel style={font=\small},
        width=0.55\columnwidth, height=4.75cm
      ]
        \addplot[red, thick]
          table[col sep=comma, x=t, y=pos_err]{data/inexact_f10_u100_euclidean_errors_control.csv};
        \addlegendentry{Euclidean}
        
        \addplot[blue, thick]
          table[col sep=comma, x=t, y=pos_err]{data/inexact_f10_u100_preconditioned_errors_control.csv};
        \addlegendentry{Precond.}
      \end{axis}
    \end{tikzpicture}
  }%\hfill
  % \\[1em]
  % --- (c) Velocity error ---
  \subcaptionbox{Velocity error $\|Rv-v_d\|$.\label{fig:vel-err}}{%
    \begin{tikzpicture}
      \begin{axis}[
        yticklabel style={font=\small},
        xticklabel style={font=\small},
        width=0.55\columnwidth, height=4.75cm
      ]
        \addplot[red, thick]
          table[col sep=comma, x=t, y=vel_err]{data/inexact_f10_u100_euclidean_errors_control.csv};
        \addplot[blue, thick]
          table[col sep=comma, x=t, y=vel_err]{data/inexact_f10_u100_preconditioned_errors_control.csv};
      \end{axis}
    \end{tikzpicture}
  }%\hfill
  \\[1em]
  %--- (d) Control magnitude ---
  \subcaptionbox{Body torque $\|\tau\|$.\label{fig:ctrl}}{%
    \begin{tikzpicture}
      \begin{axis}[
        legend pos=north east,
        legend style={font=\small},
        yticklabel style={font=\small},
        xticklabel style={font=\small},
        width=0.55\columnwidth, height=4.75cm
      ]
        \addplot[red, thick]
          table[col sep=comma, x=t, y=u_norm]{data/inexact_f10_u100_euclidean_errors_control.csv};
        
        \addplot[blue, thick]
          table[col sep=comma, x=t, y=u_norm]{data/inexact_f10_u100_preconditioned_errors_control.csv};
        
      \end{axis}
    \end{tikzpicture}
  }
  % \\[1em]
  %--- (f) Thrust ---
  \subcaptionbox{Thrust $f$.\label{fig:ctrl}}{%
    \begin{tikzpicture}
      \begin{axis}[
        legend pos=north east,
        legend style={font=\small},
        yticklabel style={font=\small},
        xticklabel style={font=\small},
        width=0.55\columnwidth, height=4.75cm
      ]
        \addplot[red, thick]
          table[col sep=comma, x=t, y=f]{data/inexact_f10_u100_euclidean_errors_control.csv};
        
        \addplot[blue, thick]
          table[col sep=comma, x=t, y=f]{data/inexact_f10_u100_preconditioned_errors_control.csv};
        
      \end{axis}
    \end{tikzpicture}
  }
  \caption{Tracking of a curve that is not integral for $X$. Errors and controls plotted as functions of time.}
  \label{fig:inexact-10-100-integral-curve}
\end{figure}

\section{Conclusions}
We introduced an update for Riemannian metrics that makes certain directions of motion cheaper to move along.
After establishing some basic completeness properties of this reward metric, we discussed the change induced on the Levi--Civita connection of such a metric.
The tensor capturing the difference between the Levi--Civita connections for the original and modified metrics enabled the formulation of a geometric preconditioner for the proportional-derivative control law for a trajectory-tracking objective.
We showcased how the preconditioned strategy performs when compared to that of \cite{Lee} on a class of underactuated maze navigation problems on $\SE(3)$.
When the nominal trajectory was chosen to be an integral curve of a spatial vector field that navigates the maze, used to deform the metric, our methodology was found overall to better track said trajectory, both in terms of body position and velocity.
In the case where the desired trajectory was not an integral curve for the vector field used to deform the metric, our method was found to underperform around and in between stretches of the trajectory with higher curvature.
A stability analysis, as well as a numerical investigation of the effects of deforming the metric with a fully roto-translational vector field, are topics for future research.
Furthermore, we plan to study the convergence properties of the metrics $g_\lambda$ from \eqref{eq:g-lambda} and investigate similar metric-modification strategies for other problems in applied differential geometry.

\end{document}